\documentclass{article}
\usepackage{amssymb,amscd,amsthm, verbatim,amsmath,color,fancyhdr, mathrsfs}
\usepackage{tikz}
\usepackage{graphicx}
\usepackage{turnstile}
\usepackage{mathtools}
\usepackage{bm}
\usepackage{inconsolata}  
\usepackage{comment}
\usepackage[colorinlistoftodos,bordercolor=orange,backgroundcolor=orange!20,linecolor=orange,textsize=scriptsize]{todonotes}

\usepackage[colorlinks=true,citecolor=blue]{hyperref}

\usepackage[all]{xy}
\usepackage{float}
\usepackage{bbold}
\usepackage{enumerate}
\usepackage{listings}

\newcommand{\RR}{\mathbb{R}}

\newcommand{\sK}{\mathsf{K}}

\newtheorem*{mainthm}{Main Theorem}
\newtheorem{thm}{Theorem}[section]

\newtheorem{lemma}[thm]{Lemma}

\newtheorem{prop}[thm]{Proposition}

\newtheorem{pro}[thm]{Problem}

\theoremstyle{definition}

\newtheorem{remark}[thm]{Remark}
\newtheorem{eg}[thm]{Example}

\DeclareMathOperator{\conv}{conv}
\DeclareMathOperator{\aff}{aff}

\usepackage{bm}

\usepackage[mathlines]{lineno}
\usepackage{fullpage}

\newcommand*\patchAmsMathEnvironmentForLineno[1]{
	\expandafter\let\csname old#1\expandafter\endcsname\csname #1\endcsname
	\expandafter\let\csname oldend#1\expandafter\endcsname\csname end#1\endcsname
	\renewenvironment{#1}
	{\linenomath\csname old#1\endcsname}
	{\csname oldend#1\endcsname\endlinenomath}}
\newcommand*\patchBothAmsMathEnvironmentsForLineno[1]{
	\patchAmsMathEnvironmentForLineno{#1}
	\patchAmsMathEnvironmentForLineno{#1*}}
\AtBeginDocument{
	\patchBothAmsMathEnvironmentsForLineno{equation}
	\patchBothAmsMathEnvironmentsForLineno{align}
	\patchBothAmsMathEnvironmentsForLineno{flalign}
	\patchBothAmsMathEnvironmentsForLineno{alignat}
	\patchBothAmsMathEnvironmentsForLineno{gather}
	\patchBothAmsMathEnvironmentsForLineno{multline}
}

\title{On colorings of hypergraphs embeddable in $\RR^d$}

\date{}
\usepackage[utf8]{inputenc}
\usepackage{authblk} 

\usepackage{array}

\author{Seunghun Lee\footnote{
		Department of Mathematics, Keimyung University, Daegu, South Korea.
    \texttt{seunghun.math@gmail.com}} 
	\ and Eran Nevo\footnote{
		Einstein Institute of Mathematics, Hebrew University, Jerusalem, Israel;  
        Mathematics Research Institute, Universidad de Valladolid, Valladolid, Spain; CMSA, Harvard University, Cambridge MA, USA. 
        Partially supported by Israel Science Foundation grant ISF-2480/20. \texttt{nevo@math.huji.ac.il}, \texttt{eran.nevo@uva.es}}
}

\begin{document}
	\maketitle
\begin{abstract}
	The \textit{(weak) chromatic number} of a hypergraph $H$, denoted by $\chi(H)$, is the smallest number of colors required to color the vertices of $H$
 so that no hyperedge of $H$ is monochromatic. For every $2\le k\le d+1$, denote by $\chi_L(k,d)$ (resp. $\chi_{PL}(k,d)$) the supremum 
$\sup_H \chi(H)$ where $H$ runs over all finite $k$-uniform hypergraphs such that $H$ forms the collection of maximal faces of a simplicial complex that is linearly (resp. PL) embeddable in $\RR^d$.

Following the program by Heise, Panagiotou, Pikhurko and Taraz, we improve their results as follows:
For $d \geq 3$, we show that A.  $\chi_L(k,d)=\infty$ for all $2\le k\le d$, B. $\chi_{PL}(d+1,d)=\infty$ and C. $\chi_L(d+1,d)\ge 3$ for all odd $d\ge 3$. 
As an application, we extend the results by Lutz and M\o ller on the weak chromatic number of the $s$-dimensional faces in the triangulations of a fixed triangulable $d$-manifold $M$:
D. $\chi_s(M)=\infty$ for $1\leq s \leq d$.
\end{abstract}
\section{Introduction}
The problem of embedding a $k$-dimensional simplicial complex in the Euclidean $d$-space, \textit{geometrically} (that is, \textit{linearly}) or \textit{piecewise-linearly} (\textit{PL} in short), is hard in general, at times undecidable; see e.g.~\cite{Hardness-of-embedding-simplicial-complexes,hardness-geometric-embeddability} for a computational complexity viewpoint on this problem. In this paper we focus on the relation between embeddability and (weak) chromatic number of simplicial complexes, equivalently of hypergraphs.
We use both known and novel suitable combinatorial constructions of families of $k$-hypergraphs with unbounded/high chromatic number, and show how to geometrically/PL embed them in $\RR^d$. The results are summarized in the Main Theorem below.

\smallskip

Given a hypergraph $H=(V,E)$, we say that $H$ is \textit{(weakly) $m$-colorable} if there is a coloring $c:V\to [m]$ such that every hyperedge of $H$ is not monochromatic. The \textit{(weak) chromatic number} of $H$, denoted by $\chi(H)$, is the smallest $m$ such that $H$ is $m$-colorable.

For every $2\le k\le d+1$, denote by $\chi_L(k,d)$ (resp. $\chi_{PL}(k,d)$) the supremum 
$\sup_H \chi(H)$ where $H$ runs over all finite $k$-uniform hypergraphs such that $H$ forms the collection of maximal faces of a simplicial complex that is geometrically (resp. PL) embeddable in $\RR^d$. Clearly $\chi_L(k,d)\le \chi_{PL}(k,d)$. Note that for every $2\le k<k' \le d+1$ we have 
$\chi_L(k',d)\le \chi_L(k,d)$. Indeed, if a pure $(k'-1)$-dimensional simplicial complex $\mathsf{K}$ geometrically (resp. PL) embeds in $\RR^d$, then so is its $(k-1)$-skeleton, and every proper coloring of its $(k-1)$-skeleton is also a proper coloring of the facets of $\mathsf{K}$.
Further, every $(k-1)$-dimensional simplicial complex geometrically embeds in $\RR^{2k-1}$, and the complete $k$-uniform hypergraph on $n$ vertices has chromatic number at least $n/(k-1)$, as no $k$ vertices can be colored the same. We obtain the trivial bound $\chi_L(k,d)=\infty$ for all
$2\le k\le (d+1)/2$.

In \cite{coloring_d-embeddable}, Heise, Panagiotou, Pikhurko and Taraz slightly improved this result by showing $\chi_L(k,d)=\infty$ also in the cases $d=2k-3$ and $d=2k-2$ for $3\le k$. 

\medskip 

Our main results are summarized in the following theorem:
\begin{mainthm}
\label{mainthm}
Fix $d\geq 3$.

A. \label{main1} $\chi_L(k,d)=\infty$ for all $2\le k\le d$. 

B. \label{main2} $\chi_{PL}(d+1,d)=\infty$.

C. $\chi_L(d+1,d)\ge 3$ for all odd $d\ge 3$.

D. $\chi_s(M)=\infty$ for every triangulable 
$d$-manifold $M$ and $1 \leq s\leq d$, where $\chi_s(M)= \sup_{H}\chi(H)$ where $H$ runs over all $(s+1)$-uniform hypergraphs such that $H$ forms the collection of $s$-dimensional faces of a simplicial complex which triangulates $M$.
\end{mainthm}

Note that Main theorem D follows from Main theorems A and B and the following recent PL extension theorem\footnote{We first learnt a more complicated proof of this PL extension theorem from Adiprasito, and then a simpler proof from Venturello and Yashfe~\cite{PL-extension-karim-geva}, essentially the one in~\cite{adiprasito-patakova2024higherdimensionalversionfarystheorem}.}~\cite{adiprasito-patakova2024higherdimensionalversionfarystheorem}, which generalizes a classical result by Bing for the $d=3$ case~\cite[Theorem I.2.A]{Bing_geometric-topology-3-manifolds}.
\begin{thm}[\cite{adiprasito-patakova2024higherdimensionalversionfarystheorem, PL-extension-karim-geva}] \label{theorem_PL_extension_karim-geva}
For every nonnegative integer $d$, if a simplicial complex $\sK$ is PL embeddable in 
$\RR^d$
then there exists a (PL) triangulation $\mathsf{T}$
of the $d$-simplex containing $\sK$ as a subcomplex in its interior
such that the boundary of $\mathsf{T}$ is combinatorially isomorphic to the boundary of the 
$d$-simplex.
\end{thm}
Indeed, consider a triangulation $\mathsf{T}$ of $M$ and a fixed facet $F$ in it. Now, the complexes illustrating the bounds in Main Theorems A and B PL embed in the interior of $F$. Hence, by Theorem~\ref{theorem_PL_extension_karim-geva}, for each of them, say $X$ corresponding to an $(s+1)$-uniform hypergraph, there exists a triangulation $\mathsf{T}(X)$ of $M$ containing $X$ as a subcomplex, obtained by subdividing the interior of $F$ appropriately.
Thus $\chi_s(M)\ge \chi(X)$, and the latter can be arbitrarily large.

\medskip 

We illustrate the new lower bounds of our Main Theorem in Table \ref{table_geometric_lower_bound}.
\begin{table}[h!]
	\centering
	\begin{tabular}{|c| c c c c c c |} 
		\hline 
	{$d$\textbackslash$k$} & {2} & {3} & {4} & {5} & {6} & {7} \\ \hline
2  & 4  & 2 & {$-$}  & {$-$} & {$-$} & {$-$}  \\
3  & {$\infty$}  & {$\infty$ \cite{coloring_d-embeddable}} & {3 (Thm C)}  & {$-$} & {$-$} & {$-$}  \\
4  & {$\infty$}  & {$\infty$ \cite{coloring_d-embeddable}} & {$\infty$ (Thm A)}  & 2 & {$-$} & {$-$}  \\ 
5  & {$\infty$}   & {$\infty$} & {$\infty$ \cite{coloring_d-embeddable}}   & {$\infty$ (Thm A)} & {3 (Thm C)} & {$-$}  \\
6  & {$\infty$}  & {$\infty$} & {$\infty$ \cite{coloring_d-embeddable}}  & {$\infty$ (Thm A)} & {$\infty$ (Thm A)} & 2 \\
7  & {$\infty$}  & {$\infty$} & {$\infty$}  & {$\infty$ \cite{coloring_d-embeddable}} & {$\infty$ (Thm A)} & {$\infty$ (Thm A)}\\
8  & {$\infty$}  & {$\infty$} & {$\infty$}  & {$\infty$ \cite{coloring_d-embeddable}} & {$\infty$ (Thm A)}  & {$\infty$ (Thm A)}\\ \hline
	\end{tabular}
	\caption{Currently known lower bounds on weak chromatic numbers for the class of $k$-uniform hypergraphs which are \textit{geometrically} embeddable in $\RR^d$. Note that when we consider PL embeddability, then all finite lower bounds other than the $d=2$ case will be changed into $\infty$ by Main theorem B.}
	\label{table_geometric_lower_bound}
\end{table}

Note that we still have a large gap between Main Theorems B and C. This motivates the following as a future research problems.

\begin{pro} \label{prob:Intro}
(1) Is $\chi_L(d+1, d)= \infty$ for all $d\geq 3$? 

(2) Let $\chi_P(d)$ be the supremum chromatic number over all hypergraphs which are the collections of maximal faces of the boundary complexes of $d$-polytopes, for a fixed $d\geq 4$. What is $\chi_P(d)$? Does $\chi_P(d)\longrightarrow\infty$ as $d\longrightarrow\infty$? 
 
\end{pro}

\textbf{Discussion of previous literature.}
It seems that previous related studies were motivated by different contexts. Probably the framework of \cite{coloring_d-embeddable} is the most similar with ours. In~\cite{coloring_d-embeddable}
lower and upper bounds on 
$\chi_L(k,d)$ and $\chi_{PL}(k,d)$ are given, 
focusing on $\chi_L(k,d)$. 
Their lower bounds are summarized in Table \ref{table_geometric_lower_bound}. Main theorems A, B and C can be considered as a continuation and improvements upon their results.

Another group of papers investigates coloring and choosability results for the \textit{face-hypergraph} of a graph embedded in a surface, regarding weak colorings \cite{coloring_surface_jctb, coloring_surface_answer, choosability_genus, Thomassen_list_coloring_2-sphere}. A face-hypergraph is constructed from a graph embedding where every face of the embedding corresponds to a hyperedge of the hypergraph consisting of the vertices incident to the face. Note that in this setting a graph embedding does not necessarily give a cell structure of a simplicial complex.
The simplicial direction is explored for general simplicial manifolds by Lutz and M\o ller \cite{coloring_surface_lutz}. They showed that (i) $\chi_2(M_g)$, where $M_g$ is a genus $g$ surface, tends to infinity as the genus $g$ tends to infinity, and (ii) $\chi_s(M)=\infty$ when $M$ is a triangulable $d$-dimensional manifold with $d\geq 3$ and $s \leq \lceil d/2\rceil$. Note that Main theorem D 
improves upon (ii). In particular, this also answers \cite[Question 6]{transversals_spheres_joseph_michael} which asks whether non 2-colorable simplicial $d$-spheres exist for $d>3$.

Our study is also related to the study of transversal numbers of geometric hypergraphs, 
studied in 
\cite{transversals_spheres_joseph_michael, novik_transversal_open_problem, cho2023transversal, novik2024transversalnumberssimplicialpolytopes} for simplicial spheres, and in \cite{nevo_stable_flag}, equivalently, for stable sets in flag spheres.
 The transversal number of a hypergraph $H$, denoted by $\tau(H)$, is the minimum size of a vertex subset which pierces every hyperedge. The transversal ratio of $H$ is the ratio $\tau(H)/|V(H)|$. Note that if $H$ is $m$-colorable then the transversal ratio of $H$ is bounded above by $\frac{m-1}{m}$. This implies that a lower bound on $\chi(H)$ is a necessary condition for corresponding lower bound on the transversal ratio of $H$.

\medskip 
\textbf{Outline.}
This manuscript is organized as follows. In Section \ref{section_codimension>=1}, we prove Main theorem A. In Section \ref{section_full_dimension-PL}, we prove Main theorem B. Finally in Section \ref{section_full_dimension-geometric}, we prove Main theorem C.

\section{\texorpdfstring{$\chi_L(k,d)$}{XL(k,d)} is unbounded when \texorpdfstring{$2 \leq k \leq d$}{2<=k<=d} and \texorpdfstring{$d\geq 3$}{d>=3}} \label{section_codimension>=1}
In this section, we prove Main theorem A. The proof is based on the family of hypergraphs defined by Ackerman, Keszegh and P\'alv\"olgyi \cite{ABAB_stabbed_pseudodisk} which the authors study 
in connection to colorings of $(AB)^l$-free hypergraphs and is extended from earlier constructions in \cite{ABABA_construction_pach, abstract_polychromatic}. We show that each such hypergraph can be geometrically embedded such that all its vertices are on the moment curve $\gamma_d(t)=(t, t^2, \dots, t^d)$. Rather interestingly, the $(AB)^l$-free property is related to this special type of embedding of the vertices into the moment curve. We first study the combinatorial criterion  for this embeddability concept in Subsection \ref{subsection_moment_curve}. In Subsection \ref{subsection_proof_Theorem A}, using the embeddability criterion, we show that the hypergraphs constructed in \cite{ABAB_stabbed_pseudodisk} are desired constructions for Theorem A.

\subsection{Combinatorial criterion for embeddability on the moment curve}
\label{subsection_moment_curve}

We say that a $k$-uniform hypergraph $H=(V,E)$ is \textit{geometrically embeddable in $\RR^d$} if there is a function $\phi:V\to \RR^d$ such that 
\begin{align}
&\dim \aff(\phi(e)) = k-1 \textrm{ for every $e\in E$, and} \label{eqn_embedding_cond_affine_ind}\\
&\conv(\phi(e_1)) \cap \conv(\phi(e_2))=\conv(\phi(e_1)\cap \phi(e_2)) \textrm{ for every $e_1, e_2 \in E$}. \label{eqn_embedding_condition}
\end{align}
Note that the inclusion $\supseteq$ in (\ref{eqn_embedding_condition}) always holds. When a $k$-uniform hypergraph $H$ is geometrically embeddable in $\RR^d$ such that the vertices of $H$ lie on $\gamma_d$ and appear on the curve in the order of $\prec$, we say that $H$ is \textit{embeddable on the moment curve $\gamma_d$ with respect to $\prec$}. 
Note that Condition (\ref{eqn_embedding_cond_affine_ind}) is redundant when we consider embedding on $\gamma_d$ as long as $k\leq d+1$.

Let $H$ be a hypergraph whose vertex set is equipped with a total order $\prec$. We say that $H$ is \textit{$l$-interlacing with respect to $\prec$} if there are distinct hyperedges $e_1$ and $e_2$ satisfying that 
\begin{align}
    &\textrm{there are vertices $v_1\prec v_2 \prec \cdots \prec v_l$ such that $v_i \in e_1$ when $i$ is odd} \nonumber\\
    	&\textrm{and $v_i \in e_2$ when $i$ is even.} \label{cond_interlacing} 
\end{align} Note that here $e_1$ and $e_2$ need not be disjoint. 

The following lemma shows a combinatorical characterization of embeddability on $\gamma_d$ using the interlacing property. This characterization was essentially introduced in  \cite{breen1973primitive} and has been used in several subsequent works dealing with geometric hypergraphs, e.g., \cite{upper_bound_pach, anshu_rectilinear}. We sketch the proof for completeness.

\begin{lemma} \label{lemma_interlacing}
	Let $k$ and $d$ be positive integers such that $k\leq d+1$. 
	For a $k$-uniform hypergraph $H$, let $\prec$ be a total order on the vertex set $V(H)$. Then, a hypergraph $H$ is embeddable on $\gamma_d$ with respect to $\prec$ if and only if $H$ is not $(d+2)$-interlacing with respect to $\prec$.
\end{lemma}
\begin{proof}
Let $\phi:V(H)\to \RR^d$ be a map into the moment curve $\gamma_d$ which respects $\prec$. 
Then condition (\ref{eqn_embedding_cond_affine_ind}) is always satisfied as points on the moment curve are in general position, thus $H$ is \emph{not} embeddable on $\gamma_d$ w.r.t. $\prec$ iff for some distinct faces $e_1, e_2\in H$ their images $\conv(\phi(e_1))$ and $\conv(\phi(e_2))$
intersect in their interiors, iff there exist disjoint subsets $f_1\subseteq e_1$ and $f_2\subseteq e_2$ such that   
$\conv(\phi(f_1))\cap\conv(\phi(f_2))$ is a single point lying in the interior of each, iff $\phi(f_1)$ and $\phi(f_2)$ have sizes summing up to $d+2$ and they form a primitive Radon partition; see~\cite[Thm.2]{Hare-Kenelly} (and also the discussion in \cite[Chap 6]{lectures_on_polytopes_book} and \cite[Lemma 3.1]{Dey_counting_triangulations}). By  \cite[Theorem]{breen1973primitive} this happens iff $f_1$ and $f_2$ interlace, iff $H$ is $(d+2)$-interlacing w.r.t. $\prec$.
\end{proof}

\subsection{Proof of Theorem A}
\label{subsection_proof_Theorem A}

For two hyperedges $A$ and $B$ in a hypergraph whose vertex set is totally ordered by $\prec$, the vertices $v_1 \prec \cdots \prec v_l$ of $A\cup B$ can be encoded as a word $w \in \{A, B, I\}^{|A\cup B|}$, where $w(i)=A$ if $v_i \in A\setminus B$, $w(i)=B$ if $v_i \in B\setminus A$, and $w(i)=I$ if $v_i \in A \cap B$ for $i\in [l]$ (here, $w(i)$ is the $i$th letter of $w$). We call $w$ the \textit{pattern of $A$ and $B$ with respect to $\prec$}. We omit the order if it is clear from the context. For a word $w$, we denote the word obtained by concatenating $w$ $k$ times by $w^k$.

\begin{eg}
	Let $A=\{v_1, v_2, v_5\}$ and $B=\{v_3,v_4, v_5\}$ where the vertices are ordered by their indices. Then the pattern of $A$ and $B$ is $AABBI$. It can be also written concisely as $A^2B^2I$. 
\end{eg}

\medskip

Now we recall the family of hypergraphs from \cite[Section 3]{ABAB_stabbed_pseudodisk} (see also \cite{ABABA_construction_pach, abstract_polychromatic}).

\smallskip 

Let $T(a,b)$ denote a full $b$-ary tree of depth $a-1$, that is, every non-leaf vertex has $b$ children and every leaf is at distance $a-1$ from the root. $H(a,b)$ is the hypergraph on the vertex set $V(T(a,b))$ whose hyperedges are of two types: either the set of all children of a non-leaf vertex of $T(a,b)$, or the vertex set of the unique path from a leaf to the root, equivalently a \textit{maximal chain} of $T(a,b)$.

\begin{lemma} \label{lemma_base_step_H(a,b)}
  When $3\leq a \leq b$, $H(a,b)$ is not $(a+2)$-interlacing with respect to any DFS (depth first search) order on $T(a,b)$.
\end{lemma}
\begin{proof}
For $T=T(a,b)$ we fix a DFS order $\prec_T$. We choose two distinct hyperedges $A$ and $B$ and consider the following exhaustive three cases. 

\smallskip

\textbf{Case I.} $A$ and $B$ are both maximal chains: The order $\prec_T$ implies that the pattern of the chains $A$ and $B$ is either of the form $I^lA^{a-l}B^{a-l}$ or $I^lB^{a-l}A^{a-l}$ for some $1\leq l \leq a-1$ ($l\ge 1$ as both $A$ and $B$ contain the root of $T$). So $\{A, B\}$ is at most $(a+1)$-interlacing (when $l=a-1$), but cannot be $(a+2)$-interlacing.

\smallskip 

\textbf{Case II.} Exactly one of $A$ and $B$ is a maximal chain: We may assume that $A$ is a maximal chain. When $A$ and $B$ intersect, the order $\prec_T$ implies that the pattern of $A$ and $B$ is in the form of $A^lB^rIA^{a-l-1}B^{b-r-1}$ where $1\leq l \leq a$ and $0 \leq r \leq b$ ($l\geq 1$ as the root of $T$ cannot be in $B$). Then $\{A,B\}$ is at most 4-interlacing, thus $\{A,B\}$ is not $(a+2)$-interlacing as $a \geq 3$. When $A$ and $B$ are disjoint, the pattern of $A$ and $B$ with respect to $\prec_T$ is $A^lB^bA^{a-l}$ for $1\leq l \leq a$,
hence $\{A,B\}$ is at most 3-interlacing.

\smallskip
		
\textbf{Case III.} Neither $A$ nor $B$ is a chain: Let $A$ and $B$ be the set of children of non-leaf vertices $v$ and $w$, respectively. 
If $w$ is a descendant of $v$, then $\prec_T$ implies that the pattern of $A$ and $B$ is $A^lB^bA^{b-l}$  where $0\leq l \leq b$. 
If $v$ is a descendant of $w$ the pattern of $A$ and $B$ with respect to $\prec_T$ is $B^lA^bB^{b-l}$ where $0\leq l \leq b$. In both cases $\{A,B\}$ is just at most 3-interlacing. Else, $v$ and $w$ are not in the same maximal chain, then the pattern of $A$ and $B$ is either $A^bB^b$ or $B^bA^b$ , hence $\{A,B\}$ is just at most 2-interlacing. In either case, $\{A,B\}$ is not $(a+2)$-interlacing for $a\geq 3$.
\end{proof}

For a positive integer $k \geq 3$, the $k$-uniform hypergraph $H^{k,m}$ and the order $\prec^{k, m}$ on $V(H^{k,m})$ are constructed recursively as follows.

\begin{itemize}
	\item We let $H^{k,2}=H(k,k)$. We fix a DFS order of $T(k,k)$ and denote it by $\prec^{k,2}$.
	
	\item Upon previously constructed $H^{k,m-1}$, we define $H^{k,m}$ as the hypergraph on the vertex set of $T=T(k, |V(H^{k,m-1})|)$ as follows. We make a copy $H^{k,m-1}_v$ of $H^{k,m-1}$ for every non-leaf vertex $v$ in $T$ where the vertices of $H^{k,m-1}_v$ are exactly the children of $v$ in $T$. The children of $v$ are ordered by $\prec^{k,m-1}_v$ which is isomorphic to $\prec^{k,m-1}$. The set of hyperedges of $H^{k,m}$ are (i) the collection of hyperedges of $H^{k,m-1}_v$ for each non-leaf vertex $v$ and (ii) the maximal chains of 
    $T$. We fix a DFS order $\prec^{k,m}$ of $T$ which is compatible with $\prec^{k,m-1}_v$ for every non-leaf vertex $v$ of $T$.
\end{itemize}

By definition, $H(k, |V(H^{k,m-1})|)$ and $H^{k,m}$ are constructed from the same tree $T=T(k, |V(H^{k,m-1})|)$, so they have the same vertex set $V(T)$. The difference is that, for each non-leaf vertex $v$ of $T$, while $H(k,|V(H^{k,m-1})|)$ takes all children of $v$ to be a single hyperedge, $H^{k,m}$ have several (smaller) hyperedges for the children of $v$, coming from $H^{k,m-1}_v$.

\medskip 

The following proposition, on colorability of these hypergraphs, is shown in \cite{ABAB_stabbed_pseudodisk}.

\begin{prop}[{\cite[Proposition 3.1]{ABAB_stabbed_pseudodisk}}] \label{prop_not-colorable}
	$H^{k,m}$ is not $m$-colorable.
\end{prop}

Next we study interlacing for these ordered hypergraphs. 
\begin{prop} \label{prop_ABABA-free}
For every $k\geq 3$ and $m\geq 2$,  
 $H^{k,m}$ is not $(k+2)$-interlacing with respect to $\prec^{k,m}$.
\end{prop}

\begin{proof}[Proof of Theorem A]
	Fix positive integers $d$, $k$ with $d\geq 3$ and $2 \leq k\leq d$. 
 The case $k=2$ is trivial, since a complete graph $K_n$ is embeddable on $\gamma_d$ for every $d\geq 3$ (and any order $\prec$ on $V(K_n)$). Let $k\geq 3$. 
 By Proposition \ref{prop_not-colorable}, it is enough to show that $H^{k,m}$ is embeddable on $\gamma_d$. By Lemma \ref{lemma_interlacing}, it is sufficient to show that $H^{k,m}$ is not $(d+2)$-interlacing. This follows directly from Proposition \ref{prop_ABABA-free} and the condition $k\leq d$.
\end{proof}

\begin{proof}[Proof of Proposition \ref{prop_ABABA-free}]
Fix a positive integer $k\geq 3$. We use induction on $m$. The base case $m=2$ follows from Lemma \ref{lemma_base_step_H(a,b)} by plugging $a=b=k$. Now, suppose $m>2$ and assume that $H^{k,m'}$ is not $(k+2)$-interlacing with respect to $\prec^{k,m'}$ for every $2\leq m'<m$.
 
     Let $A$ and $B$ be two distinct hyperedges of $H^{k,m}$. If $A$ and $B$ are contained in the same $H^{k,m-1}_v$, then by the induction hypothesis $\{A, B\}$ is not $(k+2)$-interlacing with respect to $\prec^{k,m-1}_v$.  So, $\{A, B\}$ is not $(k+2)$-interlacing with respect to $\prec^{k,m}$. Else, let $M=H(k,|V(H^{k,m-1})|)$ and $T=T(k,|V(H^{k,m-1})|)$. For a hyperedge $e$ of $H^{k,m}$ let $f(e)$ be the unique hyperedge of $M$ which contains $e$. (Thus $f(e)=e$ when $e$ is a maximal chain of $T$, and otherwise $f(e)$ strictly contains $e$.) Note that $f(A)\neq f(B)$ as $A$ and $B$ are not contained in the same $H^{k,m-1}_v$ for all vertices $v\in T$.
     By Lemma \ref{lemma_base_step_H(a,b)}, $\{f(A), f(B)\}$ is not $(k+2)$-interlacing with respect to $\prec^{k,m}$. Since $A\subseteq f(A)$ and $B \subseteq f(B)$, $\{A,B\}$ is not $(k+2)$-interlacing with respect to $\prec^{k,m}$ either. This completes the proof.
\end{proof}

\begin{remark}[Tightness] 
	In \cite{ABAB_stabbed_pseudodisk, abstract_polychromatic}, the authors studied $(AB)^l$-free hypergraphs. $(AB)^l$-free hypergraphs are defined similar to non-$(2l)$-interlacing hypergraphs: a hypergraph $H$ with fixed total order $\prec$ on its vertex set is \emph{$(AB)^l$-free} if it has no two hyperedges $A$ and $B$, and ordered vertex subset $\{v_1 \prec \dots \prec v_{2l}\}\subseteq A\cup B$ such that $v_i\in A\setminus B$ when $i$ is odd and $v_i \in B\setminus A$ when $i$ is even. Thus, the non-interlacing property is stronger. In \cite{ABAB_stabbed_pseudodisk}, it is shown that the hypergraphs $H^{k,m}$ are $ABABA$-free. Note that $H^{k,m}$ is not $(k+2)$-interlacing but is $(k+1)$-interlacing as shown in Case I of the proof of Proposition \ref{prop_ABABA-free}. 
 Note that, when $d\geq 2$, $H^{d+1,m}$ contains ``a book with three pages", namely a subhypergraph consisting of 3 hyperedges whose pairwise intersections are all equal to the same subset of size $d$; hence $H^{d+1,m}$ is not embeddable in $\RR^d$ when $d\geq 2$. 
\end{remark}

As an approach to Problem~\ref{prob:Intro}, we propose the following problem.
\begin{pro}
	Is there a family of $(d+1)$-uniform hypergraphs with unbounded chromatic numbers where each hypergraph is embeddable on $\gamma_d$?
\end{pro}

\begin{remark}
	In \cite[Sec.7]{ABAB_stabbed_pseudodisk}, the authors ask if there is a connection between Radon partitions and $(AB)^l$-freeness in higher dimensions. Lemma \ref{lemma_interlacing} describes such connection between the related interlacing property and Radon partitions. 
\end{remark}

\section{\texorpdfstring{$\chi_{PL}(d+1,d)$}{XPL(d+1,d)} is unbounded when \texorpdfstring{$d\geq 3$}{d>=3}}
\label{section_full_dimension-PL}

In this section, we prove Main theorem B.

First we recall some definitions: a hypergraph $H$ is \textit{linear} if every pair of distinct hyperedges of $H$ has intersection of size at most one. 
A simplicial (or polyhedral) complex $\mathsf{K}$ is \textit{PL embeddable} in $\RR^d$ if there is a simplicial subdivision $\mathsf{L}$ of $\mathsf{K}$ which is geometrically embeddable in $\RR^d$. The proof of Theorem B relies on the following lemma.
\begin{figure}[ht]
\centering
\includegraphics{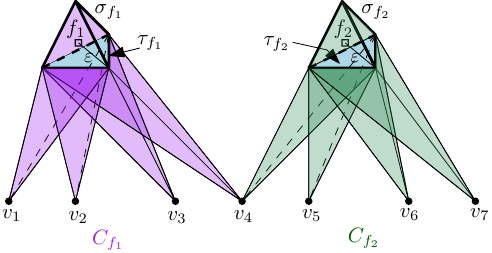}
\caption{Illustration of the construction in the proof of Lemma \ref{lemma_linear_PL-embed} when $d=3$. Here, $f_1=\{v_1, v_2, v_3, v_4\}$ and $f_2=\{v_4, v_5, v_6, v_7\}$. The tetrahedrons $\sigma_{f_i}$ are indicated by bold edges, and their respective facets $\tau_{f_i}$ are indicated by light blue.}
 \label{fig:PL}
\end{figure}

\begin{lemma} \label{lemma_linear_PL-embed}
	For $d \geq 3$, every linear $(d+1)$-uniform hypergraph is PL embeddable in $\RR^d$.	
\end{lemma}
\begin{proof}
Let $H$ be a linear $(d+1)$-uniform hypergraph. Incidences between facets and vertices in $H$ can be recorded as a bipartite graph $G$, where one color class $F$ represents the set of hyperedges of $H$, and the other $V$ represents the set of vertices of $H$. We embed $G$ into $\RR^d$ with straight lines (which is possible because $d\geq 3$) so that the vertices are in general position. 
We abuse notation and identify each vertex of $G$ with its location in $\RR^d$ and each vertex $f\in F$ with its corresponding hyperedge in $H$.
Now ``inflate" each vertex $f\in F$ into a $d$-dimensional polyhedral cell $C_f$, in the following way: place a regular $d$-simplex $\sigma_f$ of radius $\epsilon>0$ whose center of mass is $f$, and fix a facet $\tau_f$ of $\sigma_f$. Let $C_f$ be the subspace $\sigma_f\cup\bigcup_{v\in f}\conv(v\cup \tau_f)$.\footnote{In fact fixing $\tau_f$ is not important; one could replace $\conv(v\cup \tau_f)$ with $\conv(v\cup \tau_v)$ for any facet $\tau_v$ of $\sigma_f$ and the analysis of the resulted construction would work the same.}
For $\epsilon>0$ small enough, for each pair of edges $f,g\in H$, $C_f\cap C_g=f\cap g$, which is either the empty set or a single vertex shared by both, as $H$ is linear. See Figure~\ref{fig:PL} for illustration. 

It is left to show that $C_f$ has a simplicial subdivision which is combinatorially isomorphic to a subdivision of the $d$-simplex whose vertices are the vertices of $f$; compare with Hudson's
~\cite[I.Lem.1.12]{piecewise_linear_hudson} for the convex polytope case. 
To see this, we first give a simplicial subdivision $\mathsf{K}_f'$ of the boundary $\partial C_f$. This is possible because $\partial C_f$ is a compact polyhedron (see \cite[Theorem 2.11]{pl_book_rourke_sanderson} stating that \textit{any compact polyhedron is the underlying polyhedron of some simplicial complex}): the finite intersection or union of compact polyhedra is again a polyhedron \cite[1.3.(3),(5)]{pl_book_rourke_sanderson} and $\partial C_f$ can be expressed 
using finite intersections and unions of compact polyhedra.

Now, pick a point in the relative interior of  the facet $\tau_f$, and in case $\tau_f$ is contained in the boundary of $C_f$ perturb this point slightly to obtain an interior point $x$ of $C_f$, so that $C_f$ is star shaped at $x$, namely, for every point $y\in C_f$ the line segment $[x,y]$ is contained in $C_f$. Indeed, such $x$ exists since $C_f$ is the union of $d+2$ $d$-simplices all sharing the same facet $\tau_f$. We may additionally assume that $x$ and the vertices of $G$ are in general position.

We denote the vertices of $f$ by $v_1, \dots, v_{d+1}$. By the general position assumption, $\conv\{v_1, \dots, v_{d+1}\}$ is a $d$-simplex which we call $P$. We also introduce another polytope $Q$ which is differently defined by two cases: When $x \notin P$, we find an additional point $w$ so that $\conv\{v_1, \dots, v_{d+1}, w\}$ contains $x$ in its interior. Let $Q=\conv\{v_1, \dots, v_{d+1}, w\}$ in this case. When $x \in P$, we let $Q=P$. We may assume that all faces of $Q$ are simplices  in both cases.

Let $\pi_x: \partial C_f \to \partial Q$ be the radial projection with respect to $x$. Note that $\pi_x$ is well-defined and a homeomorphism because $x$ is contained in the interior of $Q$ in both cases. Then
\[\mathsf{K}_Q'=\{\sigma\cap \pi_x(\tau): \textrm{$\sigma$ is a proper face of $Q$, } \tau\in \mathsf{K}_f'\} \]
is a collection of polytopal cells whose union is  $\partial Q$. Hence $\partial Q$ has a simplicial subdivision $\mathsf{K}_Q$ which contains a subdivision of each convex cell in $\mathsf{K}_Q'$ (see \cite[Lemma 1.5]{piecewise_linear_hudson}), e.g. the derived subdivision of $\mathsf{K}_Q'$. When $Q=P$, we let $\mathsf{K}_P=\mathsf{K}_Q$. When $P \subsetneq Q$, we choose an interior point $y$ of $P$ and let $\pi_y: \partial Q \to \partial P$ be the radial projection with respect to $y$ which is again a homeomorphism. Similarly with the above, we project down $\mathsf{K}_Q$ to $\partial P$ via $\pi_y$ and make a simplicial subdivision $\mathsf{K}_P$ which is a common refinement of the boundary complex of $P$ and of $\pi_y(\mathsf{K}_Q)$.
 
 In each of the two cases, we map $\mathsf{K}_P$ back to $\partial C_f$ via the inverse map $\pi_x^{-1}$ or $(\pi_y\circ\pi_x)^{-1}$, and this gives a simplicial subdivision $\mathsf{K}_f$ of $\partial C_f$ which is a refinement of $\mathsf{K}_f'$. It is easy to see that $\mathsf{K}_f$ is combinatorially isomorphic to $\mathsf{K}_P$, and it can be extended to the whole $C_f$ and $P$ by coning the complexes over suitable interior points, either $x$ or $y$. 
 
 Thus, we described a PL embedding of $H$.
\end{proof}

A second ingredient is the Hales-Jewett theorem \cite{hales-jewett}. For positive integers $t, n$, a subset $L$ of the $n$-dimensional hypercube $[t]^n$ is called a \textit{combinatorial line} if there exist a non-empty set of indices $I = \{i_1, . . . , i_k\} \subseteq [n]$ and a choice of $a_i \in [t]$
for every $i \in [n] \setminus I$ such that
$$L = \{(x_1, \dots, x_n) \in [t]^n : x_{i_1} = \cdots = x_{i_k} \textrm{ and } x_i = a_i \textrm{ for } i \notin I\}.$$
The indices in $I$ (outside of $I$) are called the \textit{active} (\textit{fixed}, resp.) coordinates of $L$.
Let us call the hypergraph on $[t]^n$ whose hyperedges are exactly combinatorial lines of $[t]^n$ the \textit{Hales-Jewett hypergraph} on $[t]^n$. 

\begin{thm}[Hales-Jewett] \label{thm_Hales-Jewett}
	For every positive integers $t$ and $m$ there exists $n$ such that every $m$-coloring of $[t]^n$ contains a monochromatic combinatorial line. In other words, the Hales-Jewett hypergraph on $[t]^n$ is not $m$-colorable.
\end{thm}

\begin{proof}[Proof of Theorem B]
	By Lemma \ref{lemma_linear_PL-embed} and the Hales-Jewett theorem \ref{thm_Hales-Jewett} with $t=d+1$, it is enough to show that the Hales-Jewett hypergraph on $[t]^n$ is linear for every positive integers $t$ and $n$. Suppose otherwise that two distinct combinatorial lines $L_1$ and $L_2$ intersect at two vertices. Then, by comparing the values of those two vertices in each coordinate, $L_1$ and $L_2$ must have the same active and fixed coordinates, and further, at the fixed coordinates, the fixed values $a_i$ must be the same for $L_1$ and $L_2$. This means that $L_1$ and $L_2$ are determined by the same conditions, hence $L_1=L_2$, a contradiction.
\end{proof}

\begin{remark} \label{remark_more_linear_hypergraphs}
	In the literature, there are more examples of classes of linear hypergraphs which have unbounded chromatic number. For example, a \textit{balanced incomplete block design}  of order $n$, block size $k$ and index $\lambda$ (or \textit{$(n,k,\lambda)$-BIBD} in short) is a $k$-uniform hypergraph on $n$ vertices and each unordered pair of vertices is contained in exactly $\lambda$ hyperedges. It is easy to see that a $(n,k,1)$-BIBD is a linear hypergraph. It is known that for fixed $m \geq 2$, $k \geq 3$ and $\lambda$, there is a $(n,k,\lambda)$-BIBD for a suitable $n$ which is not $m$-colorable, see \cite{BIBD_coloring_Horsley-Pike, BIBD_blocksize_4, BIBD_steiner_triple}.
\end{remark}

\section{\texorpdfstring{$\chi_L(d+1,d)\geq 3$}{XL(d+1,d)>=3} for odd \texorpdfstring{$d \geq 3$}{d>=3}}
\label{section_full_dimension-geometric}
In this section, we prove Main theorem C by constructing a non 2-colorable $(d+1)$-uniform hypergraph $L_d$ which is geometrically embeddable in $\RR^d$ for every odd $d \geq 3$. 

We use abstract join and geometric join for the construction; both are denoted by the same symbol $*$, but it will be clear which join operator it means by context.
For hypergraphs $H_1, \dots, H_l$ on disjoint ground sets, the \textit{abstract join} of $H_1, \dots, H_l$ is
\[H_1*\cdots *H_l=\{e_1 \cup \cdots \cup e_l: e_i \in H_i \textrm{ for all $i\in [l]$}\}.\]
The \textit{geometric join} of subsets $S_1, \dots, S_l\subseteq \RR^d$ is defined as: 
\[S_1* \cdots * S_l=\left\{\sum_{i=1}^l \lambda_i p_i: \sum_{i=1}^l \lambda_i=1, \lambda_i\geq 0, \textrm{ and } p_i \in S_i \textrm{ for every $i \in [l]$}\right\}.\]
We say that subsets $S_1, \dots, S_l\subseteq \RR^d$ are in \textit{skewed position} if
\[\dim\left(\aff\left(\bigcup_{i=1}^lS_i\right)\right)=l-1+\sum_{i=1}^l \dim(\aff(S_i)).\]
Note that every $k$ disjoint edges in the boundary of the cyclic polytope $C_{2k}(n)$ are in skewed position, 
and their images under the projection from $\RR^{2k}$ to $\RR^{2k-1}$ that forgets the last coordinate are in skewed position, with vertices on $\gamma_{2k-1}$, and with geometric join a $(2k-1)$-simplex.

\subsection{Construction of \texorpdfstring{$L_d$}{Ld}}
First, we construct $L_d$. Fix a positive integer $k \geq 2$ and set $d=2k-1$.

\begin{enumerate}[(1)]
	\item For $i\in [d]=[2k-1]$, let
	$$V_i=\{v_1^i, \dots, v_{4k-2}^i\}$$
	be a set of $4k-2$ vertices. Let $P_i$ be the edge set of the path on $V_i$ where the vertices appear in order, that is,
	$$P_i=\{\{v_1^i,v_2^i\}, \dots, \{v_{4k-3}^i,v_{4k-2}^i\}\}.$$
	
	\item Let
	$$K_1=\bigcup_{1\leq i_1<\cdots <i_k\leq d}P_{i_1} * \cdots * P_{i_k},$$
	where $*$ denotes the abstract join operation.
	Note that $K_1$ is a $(d+1)$-uniform hypergraph.
	
	\item We define $K_2$ in a similar way. Let
	\begin{align*}
	W_i&=\{w_1^i, \dots, w_{4k-2}^i\}, \textrm{ and} \\
	Q_i&=\{\{w_1^i,w_2^i\}, \dots, \{w_{4k-3}^i,w_{4k-2}^i\}\}
	\end{align*}
	for every $i \in [d]$. Then, let
$$K_2=\bigcup_{1\leq i_1<\cdots <i_k\leq d}Q_{i_1} * \cdots * Q_{i_k}.$$
	
	\item Let
	\begin{align*}
	M_1&=\{\{v_1^i, v_3^i, \dots, v_{2k-1}^i\}, \{v_{2k}^i, v_{2k+2}^i, \dots, v_{4k-2}^i\}: i \in [d]=[2k-1]\}, \textrm{ and}\\
	M_2&=\{\{w_1^i, w_3^i, \dots, w_{2k-1}^i\}, \{w_{2k}^i, w_{2k+2}^i, \dots, w_{4k-2}^i\}: i \in [d]=[2k-1]\}.
	\end{align*}
	That is, $M_1$ is the collection whose elements are the subset of $V_i$ which consists of the vertices with the first $k$ odd indices and the subset of $V_i$ which consists of the vertices with the last $k$ even indices for every $i\in [d]$. $M_2$ is defined similarly. Note that $M_1$ and $M_2$ are $k$-uniform hypergraphs. 
	
	\item Finally, we define
	$$ L_d=K_1 \cup K_2 \cup M_1*M_2. $$
	Note that $L_d$ is a $(d+1)$-uniform hypergraph.
\end{enumerate}
\subsection{Non-2-colorability of \texorpdfstring{$L_d$}{Ld}}
\begin{prop} \label{prop_Ld_not-2-colorable}
Let $d\ge 3$ be odd. Then 	
 $L_d$ is not 2-colorable.
\end{prop}
\begin{proof}
	Suppose otherwise and we color all vertices of $L_d$ by RED and BLUE so that no hyperedge is monochromatic. In $K_1$, note that not all $2k-1$ vertex sets $V_i$ have consecutive vertices of the same color: Otherwise, without loss of generality, we can assume that RED was used in $V_1, \dots, V_k$ in consecutive vertices. By definition of $K_1$, this gives a hyperedge of $K_1$ which is monochromatic, a contradiction. So at least one $V_i$ has an alternating coloring, say $V_1$ does. The same also holds for $K_2$, so at least one $W_j$ has an alternating coloring, say $W_1$ does. By the alternating colorings given to $V_1$ and $W_1$, at least one (in fact, there are exactly two) of the 4 hyperedges of $M_1 * M_2$ which are subsets of $V_1\cup W_1$ should be monochromatic. \end{proof}

Note that $L_d$ is 3-colorable: For each $i \in [d]$, we color the vertices of $P_i$ by RED, BLUE and GREEN alternatingly in order. We do the same for every $Q_i$. It is easy to see that each of the hyperedges in $K_1$, $K_2$, and $M_1*M_2$ has at least two colors from this coloring.

\subsection{Embeddability of \texorpdfstring{$L_d$}{Ld}}
In this subsection we prove the following proposition. Recall that $L_d$ is defined only for odd $d\geq 3$.
\begin{prop} \label{prop_Ld_embeddable}
For every odd $d\ge 3$, 	
	$L_d$ is geometrically embeddable in $\RR^d$. 
\end{prop}

The following lemma is useful in the proof (see also \cite[Section 6.1]{triangulations_book}).
\begin{lemma}[{\cite[Lemma 2.3]{edelman_cyclic_triangulation_envelope}}]
	\label{lemma_cyclic_triangulation_envelope}
Let $p_1, \dots, p_n$ be the vertices of $C_{d+1}(n)$ indexed by their order along the moment curve $\gamma_d$.
Let $\hat{0}_{n,d}$ be the set of facets in the lower envelope of $C_{d+1}(n)$, namely the facets of the boundary of $C_{d+1}(n)$ which are visible from any point on the $x_{d+1}$-axis with sufficiently small $(d+1)$-coordinate. Then,
\begin{align*}
\hat{0}_{n,d}=\left\{
\begin{array}{lr}
\{p_1\}*D(d/2, \{p_2, \dots, p_n\}) & \text{when $d$ is even,}\\
D((d+1)/2, \{p_1, \dots, p_n\}) & \text{when $d$ is odd,}
\end{array}\right.
\end{align*}
where 
$$D(k, \{a_1, \dots, a_n\})=\{\{a_{i_1},a_{i_1+1}, \dots, a_{i_k}, a_{i_k+1}\}: i_{j+1}\geq i_j+2 \textrm{ for $j\in [k-1]$}\}$$
for an ordered set $\{a_1, \dots, a_n\}$. Furthermore, $\hat{0}_{n,d}$ gives a triangulation of $C_d(n)$.
\end{lemma}
 
\medskip
\begin{figure}[ht] 
\centering
\includegraphics{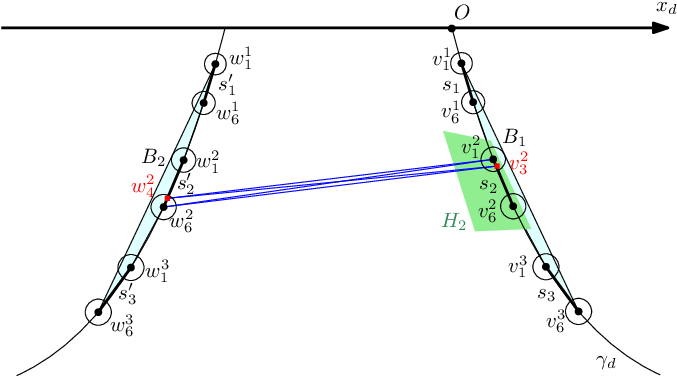}
\caption{Illustration of the geometric embedding of $L_3$. A small circle around each end vertex of $s_i$ and $s_i'$ describes a small neighborhood where the other vertices are introduced. The blue lines indicate a thin tetrahedron in the realization of $M_1*M_2$. The other notations are introduced in the proof below.}
\label{fig:Ld_embedding}
\end{figure}

We first overview how we geometrically embed $L_d$ in $\RR^d$. Recall $d=2k-1$. For $K_1$, we begin with the cyclic polytope $C_d(2d)$, denoted by $B_1$. The vertices of $B_1$ are exactly the end vertices $v_1^i$ and $v_{4k-2}^i$ of the paths $P_i$ for $i\in [d]$, ordered $v_1^1,v_{4k-2}^1,v_1^2,v_{4k-2}^2\ldots,v_1^d,v_{4k-2}^d$ on $\gamma_d$; internal vertices of $P_i$ are not introduced yet. Denote the line segment $[v_1^i, v_{4k-2}^i]$ by $s_i$. 
It follows from  Lemma \ref{lemma_cyclic_triangulation_envelope} that the geometric joins, which equal the convex hulls, of $s_{i_1}, \dots, s_{i_k}$, for all $k$-subsets $\{i_1, \dots, i_k\}\subseteq [d]=[2k-1]$ form a subcollection of the facets in the triangulation $\hat{0}_{2d,d}$ of $B_1$ . Hyperedges of $K_1$ will eventually appear in a subdivision of this triangulation after a perturbation of the vertices. In this way we will embed $K_1$.

Then $B_2$ is defined as a reflected copy of $B_1$ w.r.t. some hyperplane $H$. Denote by $s'_i$ the reflected image of $s_i$. A crucial assumption is that $B_1$ and $B_2$ are far enough from each other so that there exist hyperplanes $H_i$ (resp. $H'_i$) which weakly separate $B_1$ and $B_2$, and $B_1 \cap H_i=s_i$ (resp. $B_2 \cap H'_i=s'_i$) for every $i\in [d]$. 
These hyperplanes will be used to separate hyperedges in $K_1$ (resp. $K_2$) from  hyperedges in $M_1*M_2$.

In $B_1$, we subdivide each $s_i$ to introduce internal vertices of $P_i$, and similarly in $B_2$. We perturb these points so that hyperedges of $M_1*M_2$ become full-dimensional and we can still use the hyperplane separators we previously constructed. Also, we assume that the first half (resp. last half) of the internal vertices of $P_i$ were chosen very close to $v_1^i$ (resp. to $v_{4k-2}^i$), and we do the same for $B_2$. This is to ensure that every two disjoint hyperedges of $M_1*M_2$ do not  intersect in the embedding. We give some additional assumptions for the other pairs of hyperedges in $M_1*M_2$.  
Thus, these final positions will provide a geometric embedding of $L_d$. We now proceed to the details.

\begin{proof}[Proof of Proposition \ref{prop_Ld_embeddable}]
The construction of the geometric embedding progresses through several steps. For notational simplicity, we use the same notation for the vertices of $L_d$ and their embedded images in $\RR^d$ even when their locations are modified, for instance, by perturbing them.
	
	\medskip 

	\textbf{Step 1. Constructing the 
 polytope $B_1$.} We first construct the polytope $B_1$  for $K_1$ in $\RR^d$ which has the combinatorial type of $C_d(2d)$. Let $\gamma_d$ be the moment curve in $\RR^d$. We choose $2d$ vertices
	$$v_1^1, v_{4k-2}^1, v_1^2, v_{4k-2}^2, \dots,  v_1^d, v_{4k-2}^d$$
	in that order on $\gamma_d$. We denote the set of those vertices by $V'$ and let $B_1=\conv(V')$. Let
	$$H_{V'}=\left\{\{v_1^{i_1}, v_{4k-2}^{i_1},v_{1}^{i_2}, v_{4k-2}^{i_2},\dots, v_1^{i_k}, v_{4k-2}^{i_k}\}:1\leq i_1<\cdots <i_k\leq d\right\}.$$
	One can easily see that $H_{V'} \subseteq D((d+1)/2, V')$, hence by Lemma \ref{lemma_cyclic_triangulation_envelope},
	$H_{V'}$ forms a part of this triangulation of $B_1$, and in particular, $H_{V'}$ is embedded on the moment curve.
	
	\medskip 
	
\textbf{Step 2. Making a reflected copy $B_2$ of $B_1$ and separating hyperplanes $H_i$ and $H_i'$.}
Now, we want to construct the polytope $B_2$ similarly for $K_2$ and suitably locate $B_1$ and $B_2$. Recall that we denote the segment $[v_1^i, v_{4k-2}^i]$ by $s_i$ for every $i\in [d]$. 

By Lemma \ref{lemma_cyclic_triangulation_envelope} 
every $s_i$ belongs to the lower envelope of $B_1$, that is, $s_i$ is visible from the point
$p_i=(0,\dots, 0, z_i)$ for every negative $z_i$ with large enough absolute value $|z_i|$.
For each $i \in [d]$ we choose a facet $f_i$ in the lower envelope of $B_1$ which contains $s_i$. Then, by slightly rotating $\aff(f_i)$ while containing $s_i$, we obtain a supporting hyperplane $H_i$ of $B_1$ at $s_i$ (so $B_1 \cap H_i = s_i$) which satisfies the extra property that it separates $B_1$ from $p_i$.

Let $B_2$ be a reflected copy of $B_1$ with respect to the hyperplane 
\[P_z=\{(x_1, \dots, x_{d-1}, z): x_1,\dots, x_{d-1}\in \RR \}\]
for a fixed $z \in \RR$. We denote the vertices of $B_2$ with the same notation 
of $K_2$. That is, the vertex set of $B_2$ is
$$W'=\{w_1^1, w_{4k-2}^1, w_1^2, w_{4k-2}^2, \dots,  w_1^d, w_{4k-2}^d\},$$
where each $w^j_l$ is the reflection of $v^j_l$ w.r.t. $P_z$. 
Let $s_i'=\conv(\{w_1^i, w_{4k-2}^i\})$ and
$$H_{W'}=\left\{\{w_1^{i_1}, w_{4k-2}^{i_1},w_{1}^{i_2}, w_{4k-2}^{i_2},\dots, w_1^{i_k}, w_{4k-2}^{i_k}\}:1\leq i_1<\cdots <i_k\leq d\right\}.$$
We denote the supporting hyperplane for $s_i'$ by $H_i'$, which are chosen as the reflection of $H_i$ by $P_z$. We assume that $z$ is chosen to be negative with sufficiently large $|z|$ so that every $H_i$ separates $B_1$ and $B_2$ such that $B_2$ is in an open halfsapce of $H_i$ for every $i \in [d]$. By symmetry, $H_i'$ also separtates $B_1$ and $B_2$ such that $B_1$ is in an open halfspace of $H_i'$ for every $i \in [d]$.

\medskip 

\textbf{Step 3. Introducing remaining vertices of $L_d$ in the embedding via general position assumptions.} 
We subdivide each $s_i$ by introducing $4k-4$ new vertices $v_2^i,  v_3^i, \dots,  v_{4k-3}^i$ so that $v_1^i, v_2^i,  \dots,  v_{4k-3}^i, v_{4k-2}^i$ are  in order along $s_i$. 
We also subdivide $s_i'$ using $w_1^i, \dots, w_{4k-2}^i$ in order. This subdivision of line segments induces a  subdivision of the simplices 
$$s_{i_1}* \cdots * s_{i_k}, \textrm{ and } s_{i_1}'* \cdots * s_{i_k}'$$
where $1 \leq i_1< \cdots < i_k \leq d$, which are exactly the simplices in the geometric realization of $H_{V'}$ of Step 1. Therefore, this gives a geometric realization of $K_1$ and $K_2$. Denote the $\epsilon$-neighborhood of $p \in \RR^d$ by $N_\epsilon(p)$. On $s_i$, we first locate the vertices $v_2^i, v_3^i, \dots, v_{2k-1}^i$  in $N_\epsilon(v_1^i)$ and locate $v_{2k}^i, \dots, v_{4k-3}^i$ in $N_\epsilon(v_{4k-2}^i)$, where $\epsilon>0$ is chosen such that the neighborhoods $N_{\epsilon}(v)$, over all vertices of $B_1$ and  $B_2$, are pairwise disjoint, and further, for every four distinct vertices $u,v,w,x$ of $B_1$ and $B_2$, the joins $N_\epsilon(u)*N_\epsilon(v)$ and $N_\epsilon(w)*N_\epsilon(x)$ are disjoint. (These joins  are ``thin thickening" of the segments $uv$ and $wx$ resp.) 
This is possible after we slightly perturb the vertices of $B_1 \cup B_2$ so that they are in general position.

We also perturb the hyperplanes $H_i$ and $H_j'$ for each $i,j \in [d]$ as well as the vertices of $B_1 \cup B_2$ if necessary to fix $(k-1)$-dimensional flats $F_i$ and $F_j'$  such that 
$F_i$ and $F_j'$ are in skewed position, 
 $$s_i \subseteq F_i \subseteq H_i \textrm{ and } s_j' \subseteq F_j' \subseteq H_j',$$
 and $H_i$ and $H_j'$ satisfies the separating property described above. We further assume that for every distinct $w_1, w_2 \in W'$ (resp. $v_1, v_2 \in V'$) and every $F_i$ (resp. $F_j'$), $F_i \cup \{w_1, w_2\}$ (resp. $F_j' \cup \{v_1, v_2\}$) is in general position. By this, we can require $\epsilon$ to be sufficiently small so that for every $F_i$ (resp. $F_j'$) and $w_1, w_2 \in W'$ (resp. $v_1, v_2 \in V'$), there is a hyperplane which contains $F_i$ (resp. $F_j'$) and strictly separates $N_\epsilon(w_1)$ from $N_\epsilon(w_2)$ (resp. $N_\epsilon(v_1)$ from $N_\epsilon(v_2)$).  Clearly this is possible.

Now, for every $i\in [d]$, we perturb the new vertices $v_2^i, v_3^i, \dots, v_{2k-1}^i$ within $N_\epsilon(v_1^i)$ and perturb $v_{2k}^i, \dots, v_{4k-3}^i$ within $N_\epsilon(v_{4k-2}^i)$ as follows: 
For the odd-indexed vertices $O_i=\{v_1^i, v_3^i, \dots, v_{2k-1}^i\}$, we additionally require that they are only perturbed within the flat $F_i$ and are in general position, which implies that the resulting vertices affinely span $F_i$. We require the same for the even-indexed vertices $E_i=\{v_{2k}^i, v_{2k+2}^i, \dots, v_{4k-2}^i\}$. Note that the sets $O_i$ and $E_i$, for all $i\in [d]$,  are exactly the hyperedges of $M_1$. The other vertices on $s_i$ are perturbed so that they are contained in $H_i^+$, the open halfspace which contains the interior of $B_1$. This perturbation should be small enough so that it also gives a geometric embedding of $K_1$  (note that for any simplicial complex $K$ geometrically embedded in $\RR^d$, any sufficiently small perturbation again yields a geometric embedding since for any two simplices $\sigma_1, \sigma_2$ of $K$ and the hyperplane $H$ which (weakly) separates $\sigma_1$ and $\sigma_2$, the separation is preserved after the perturbation by a small consistent perturbation of $H$, namely, a perturbation where $H$ continues to contain the perturbed points which were originally on $H$). The vertices $w_j^i$ are perturbed to satisfy the same properties w.r.t. the flats $F_i'$ and halfspaces $H_i'^{+}$, so they give a geometric embedding of $K_2$. 
Finally, we add the simplices of $M_1*M_2$. This completes the construction.

\medskip

Now we show that our construction gives an embedding. Condition (\ref{eqn_embedding_cond_affine_ind}) is easily satisfied: Hyperedges of $K_1$ and $K_2$ are 
in general position; 
hyperedges of $M_1*M_2$ are also in general position
by the skewedness of each pair $F_i$ and $F'_j$ and as hyperedges in $M_1$ and $M_2$ are 
spanning their corresponding flats. 

\smallskip 

We now show that Condition (\ref{eqn_embedding_condition}) is satisfied. Let $e_1$ and $e_2$ be two distinct hyperedges of $L_d$. There are three cases to consider: (i) $e_1$ and $e_2$ are from $K_1\cup K_2$. When they are from the same $K_i$ we already showed that our construction gives an embedding of $K_i$. Else, without loss of generality $e_i\in K_i$ for $i\in [2]$. Since $K_1$ and $K_2$ are separated by a hyperplane then the realizations of $e_1$ and $e_2$ do not intersect.

\smallskip

(ii) $e_1\in K_i$ and $e_2\in M_1*M_2$; without loss of generality $i=1$.  Then, $e_2 = f_1\cup f_2$ where $f_j \in M_j$ for $j\in [2]$. From the construction, we can find 
the hyperplane $H_l$ which contains $f_1$, and the end vertex $v$ of the segment $s_l$ such that $f_1$ is contained in $N_\epsilon(v)$. 
Note that $f_2 \subseteq H_l^-$ ($H_l^-$ is the other open halfspace than $H_l^+$ bounded by $H_l$), hence $e_2$ is realized in the closed halfspace $H_l\cup  H_l^-$. 
On the other hand, $e_1$ is realized in the closed halfspace 
$H_l\cup  H_l^+$. 
Thus,
\begin{align*}
&\conv(e_1)\cap \conv(e_2) = (\conv(e_1)\cap H_l)\cap (\conv(e_2)\cap H_l)\\
= &\conv(e_1 \cap H_l) \cap \conv(e_2 \cap H_l)=\conv(e_1 \cap H_l) \cap \conv(f_1).
\end{align*}
But $e_1 \cap V_l$, a superset of $e_1 \cap H_l$, is either empty or of size 2. In the former case, we see that $\conv(e_1)\cap \conv(e_2)$ is empty. In the latter case, $e_1 \cap H_l$ consists of either (a) a single vertex $w$ in $O_l\cup E_l$ or (b) two vertices $v^l_{2k-1}$ and $v^l_{2k}$. 

In Case (a),  \[\conv(e_1 \cap H_l)\cap \conv(f_1)=\{w\}=\conv(e_1 \cap f_1)=\conv(e_1\cap e_2)\] 
 when $w\in f_1$, and
\begin{align*}
&\conv(e_1 \cap H_l)\cap \conv(f_1)\subseteq \conv(O_l)\cap \conv(E_l)\\
\subseteq &N_\epsilon(v^l_1)\cap N_\epsilon(v^l_{4k-2}) =\emptyset = \conv(e_1 \cap e_2)
\end{align*}
 when $w \notin f_1$.

In Case (b), if our perturbation of $V$ was sufficiently small, $\conv(e_1\cap H_l)\cap \conv(f_1)$ is a single vertex, $v^l_{2k-1}$ when $f_1=O_l$ and $v^l_{2k}$ when $f_1=E_l$; this is possible by applying an argument similar with one given at Step 3, for preserving a weak separation under a small perturbation, to $\conv(e_1\cap H_l)$ and $\conv(f_1)$. Hence we have $\conv(e_1\cap H_l)\cap \conv(f_1)=\conv(e_1\cap e_2)$.

\smallskip

(iii) Finally, $e_1, e_2\in M_1*M_2$.
There exist unique partitions $e_1=f_1^1\cup f_2^1$ and $e_2=f_1^2 \cup f_2^2$, where $f_1^1, f_1^2 \in M_1$ and $f_2^1, f_2^2 \in M_2$. 
If these four sets $f_j^i$ are distinct, then each is contained in $N_\epsilon(v)$ for a distinct vertex $v\in L_d$, and as the joins $N_\epsilon(u)*N_\epsilon(v)$ and $N_\epsilon(w)*N_\epsilon(x)$ are disjoint for every distinct vertices $u,v,w,x\in L_d$ we conclude that $\conv(e_1)\cap \conv(e_2)$ is empty. 

Else, without loss of generality, $f_1^1=f_1^2$; we find the flat $F_i$ which contains $f_1^1$, and the vertices of $W'$ corresponding to $f_2^1$ and $f_2^2$, say $w_1$ and $w_2$. 
Now, there exists a hyperplane $H''$ containing $F_i$ which separates $N_\epsilon(w_1)$ from $N_\epsilon(w_2)$, hence $H''$ separates $f_2^1$ from $f_2^2$. If follows that   
the intersection $\conv(e_1)\cap \conv(e_2)$ is contained in $H$ and equals $\conv(f_1^1)=\conv(e_1\cap e_2)$. 
 
This completes the proof.
\end{proof}

\section*{Acknowledgement}
We thank Karim Adiprasito and Geva Yashfe for their illustrations of Theorem \ref{theorem_PL_extension_karim-geva}, and anonymous referees for their helpful remarks and comments. This project was initiated when the first author was a postdoctoral research fellow at Hebrew University of Jerusalem, Israel.

\bibliographystyle{hplain}
\bibliography{bibliography}

\end{document}